\documentclass[10pt,twoside]{article}

\usepackage{7OSME-style}

\usepackage{xfrac}

\newtheorem{conjecture}{Conjecture}
\newtheorem{lemma}{Lemma}
\theoremstyle{definition}
\newtheorem{definition}{Definition}
\newtheorem*{definition*}{Definition}
\newtheorem{openq}{Open Question}

\newcommand{\im}{\mathrm{Im}}

\title{Knot Embeddings in Improper Foldings\footnote{
\emph{Note for arXiv verison:} This paper originally appeared in the 7OSME proceedings and has since received a few very minor corrections. Please cite as below, optionally including the arXiv identifier.\protect\\\protect\\{\hangindent=0.4cm\hangafter=2\textsf{\scriptsize Slote, J., \& Bertschinger, T. (2018). Knot Embeddings in Improper Foldings. In R. J. Lang, M. Bolitho, \& Z. You (Eds.), \emph{Origami\textsuperscript{7}: the proceedings from the 7\textsuperscript{th} International Meeting on Origami in Science, Mathematics, and Education (7OSME)} (Vol. 2: Mathematics, pp. 451–464). St Albans, United Kingdom: Tarquin.}
}}}
\authorlist[Slote, Bertschinger]{Joseph Slote, Thomas Bertschinger} 
\affiliations{Joseph Slote 
\at California Institute of Technology, 1200 E. California Blvd., Pasadena, CA 91125,
\email{jslote@caltech.edu}
\vspace{-1.3em}
{\and\footnotesize  \at (Previously Carleton College, 1 North College Street,
Northfield, MN 55057, \email{joseph.slote@gmail.com})}\\
\and Thomas Bertschinger \at Carleton College,
1 North College Street,
Northfield, MN 55057, \email{tahbertschinger@gmail.com}}

\makeatletter
  \let\runtitle\@title
  \let\runauthor\shortauthor
\makeatother

\begin{document}
\maketitle

\begin{abstract}
Simple closed curves in the plane can be mapped to nontrivial knots under the action of origami foldings that allow the paper to self-intersect.
We show all tame knot types may be produced in this manner, motivating the development of a new knot invariant, the \emph{fold number}, defined as the minimum number of creases required to obtain an equivalent knot.
We study this invariant, presenting a bound on the fold number by the diagram stick number as well as a class of torus knots with constant fold number.
We also pursue a characterization of those foldings which admit nontrivial knots, giving a proof that no ``physically realizable'', or \emph{proper}, foldings can admit nontrivial knots.
A number of questions are posed for further study.
\end{abstract}

\section{Introduction}
\label{sec:introduction}

It is not immediately clear that a connection ought to exist between the mathematics of origami and knot theory: folding in real life is, after all, the manipulation of paper by isotopies.
Yet the wealth and variety of existing knot invariants along with the combinatorial nature of origami crease patterns indicate that, if a connection could be made, the structure of crease patterns may provide an interesting characterization of the complexity of knots.
It is our aim in this work to show that such a connection is not only possible but maybe even natural---and that it opens a new avenue of interesting study for topologists and mathematical origamists alike.
This connection was first suspected by Jacques Justin in \cite{Justin1997TowardOrigami}.
We will begin by presenting a definition of origami folding which admits self-intersection, proceed to exploring the knots thereby created, and conclude with an investigation into which maps admit knots.

A variety of mathematical formalisms for origami folding have appeared in the literature.
Perhaps the earliest are the Huzita-Hatori or Huzita-Justin Axioms for flat-folding (e.g., \cite{Alperin2006One-Axioms}, \cite{Justin1991ResolutionGeometriques}), which describe the possible alignments which can practically be used to define creases in two-dimensional origami.
While these are important definitions for understanding constructibility and flat-folded origami, they serve a different purpose than the topological definition we will use here.
In particular, we will not be concerned with ``constructability'' of crease patterns related to our foldings, nor are we interested in any sort of folding sequence or even the rigid-foldability of the maps (though these would be worthwhile extensions to the topic).

Instead our work will focus on the topological properties of a folding as described by a map from the unit square to the ambient space of $\mathbb R^3$.
This definition bears greater similarity to that used in the computational geometry work of Demaine, O'Rourke, and others (e.g., \cite{Demaine2007GeometricPolyhedra}) but includes the additional allowance that the paper may self-intersect.

\begin{definition}
An \emph{origami folding} is a piecewise-linear arcwise isometry $[0,1]^2\to \mathbb{R}^3$.
\end{definition}

Let $F$ be an origami folding.
Here by ``piecewise-linear'' we mean that $\im(F)$ is a union of polygons in $\mathbb{R}^3$.
Piecewise-linearity entails there there is a graph embedding $G\subset [0,1]^2$ such that $F$ is non-differentiable precisely on $G$.
We refer to $G$ as the \emph{crease pattern} for $F$.
By ``arcwise isometry'' we mean that if $p:[0,1]\to [0,1]^2$ is a path in the unit square then its Euclidean arc length, $L(p)$, is preserved by the folding: $L(p) = L(F(p))$ (e.g., \cite{Burago2001AGeometry}).
Note that in what follows we will frequently use that fact that arcwise isometry is equivalent to piecewise isometry for piecewise-linear maps.

According to this definition $F$ may crease the paper and create self intersection, but $F$ may not stretch, shrink, or change the curvature of the paper.
The unit square was chosen as the domain of foldings in the spirit of origami, but other domain surfaces (orientable or not) may be worthy of further study.
Note that in some of the constructions below we use rectangles for the folding domain in the name of visual clarity.
This is without loss of generality, however, as in all cases a square may be used instead by folding away the extra paper or scaling the construction appropriately.

\begin{definition}
A knot $K$ is an \emph{origami knot} if there exists a piecewise-linear loop $\ell: S^1\hookrightarrow [0,1]^2$ on the origami paper and an origami folding $F$ such that $F(\ell) = K$.
When this property is satisfied, we say \emph{F} \emph{admits} $K$.
\end{definition}

As the piecewise-linear restriction suggests, we are concerned only with tame knots here and hence by ``knot'' we will mean ``tame knot''.
Further, in a number of diagrams the loop $\ell$ will appear to be drawn with smoothly curving sections;
this is again for the sake of visual clarity.
The cautious reader may imagine $\ell$ to be made of piecewise-linear segments too short to be resolved by the human eye.
See Figure \ref{trefoil-example} for a construction of a trefoil as an origami knot.

\begin{figure}
\centering
\includegraphics[width=.8\textwidth]{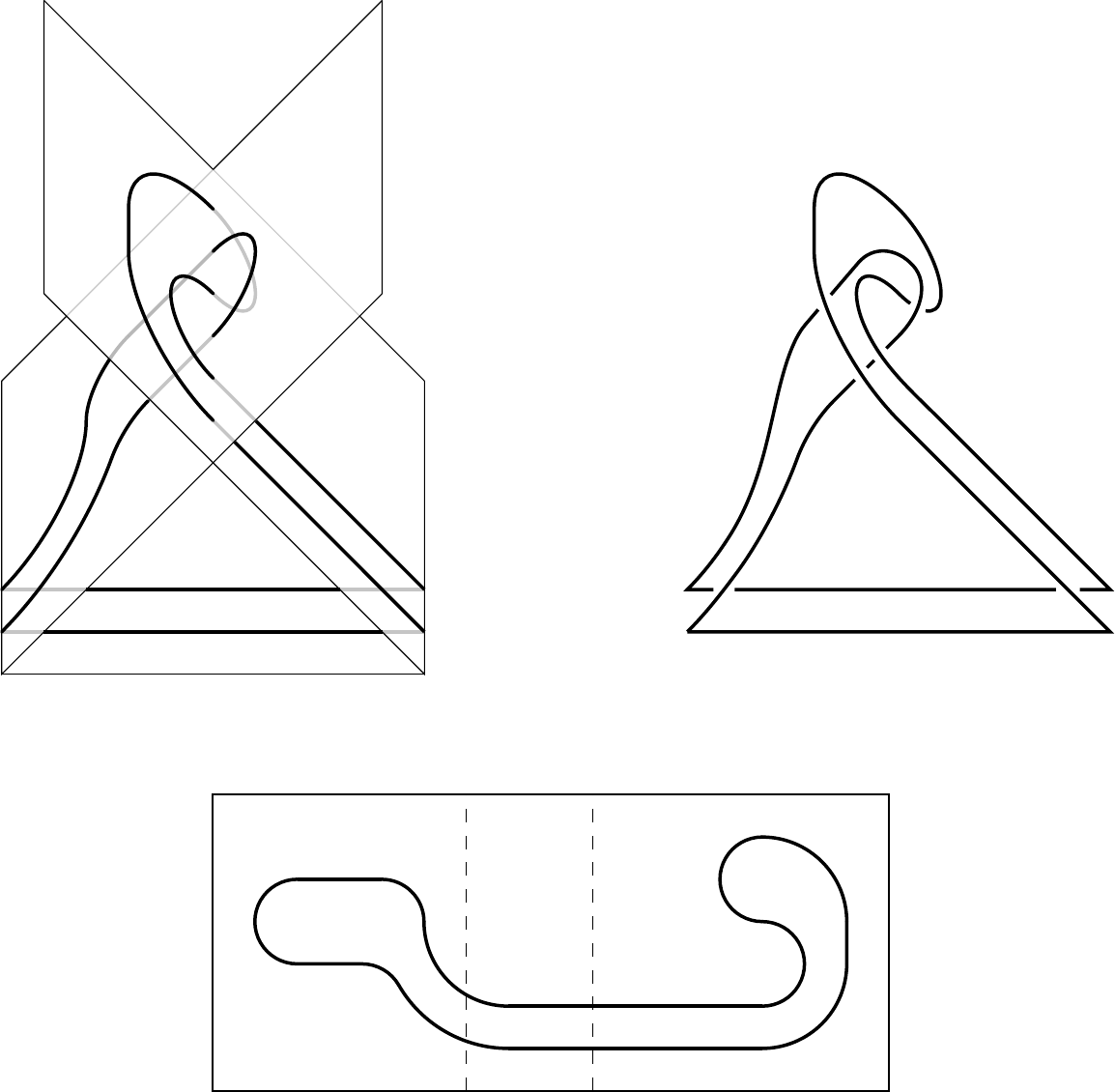}
\caption{(Left) A folding with an embedded loop; (Right) The loop is mapped to the trefoil under the folding; (Bottom) The crease pattern for this folding.}
\label{trefoil-example}
\end{figure}

\begin{definition}
The \emph{fold number} of a knot $K$ is the minimum number of edges in a crease pattern of an origami folding which admits an equivalent knot $K'$.
\end{definition}

We now explore this invariant:
how is the fold number related to other knot invariants?
What is the fold number of various families of knots?
Can we characterize the set of foldings which admit knots?
What is its structure?

\section{The Fold Number}
\label{sec:section}

Recall the \emph{diagram stick number} of a knot $K$ is the minimum number of line segments required to create a piecewise-linear knot diagram for $K$.

\begin{theorem}
\label{thm:allknots}
All knots are equivalent to an origami knot.
Moreover, the diagram stick number of a knot is an upper bound on the fold number.
\end{theorem}
\begin{proof}
Let $s$ be a minimal stick diagram (with crossing information omitted) for a given knot $K$ and let $n$ be the stick number of $K$.
We will construct an origami folding $F$ and a loop $\ell$ such that $F(\ell) = s$ (sans crossing information), after which small perturbations of $\ell$ yield the desired over- and under-crossings.

Begin by assigning an orientation to $s$.
Note that $s$ is piecewise-linear, so we may treat it as a directed graph for the purpose of labeling its edges and vertices.
Label its edges $e_1, e_2, \ldots, e_n$ such that $e_n$ is adjacent to $e_1$ and assign to each vertex $v$ the index of its incoming edge.
Finally, position a point $q$ in a region of the $xy$-plane bounded by $s$.
An example is pictured in Figure \ref{labeling-shadow}.

\begin{figure}
\centering
\includegraphics[width=.9\textwidth]{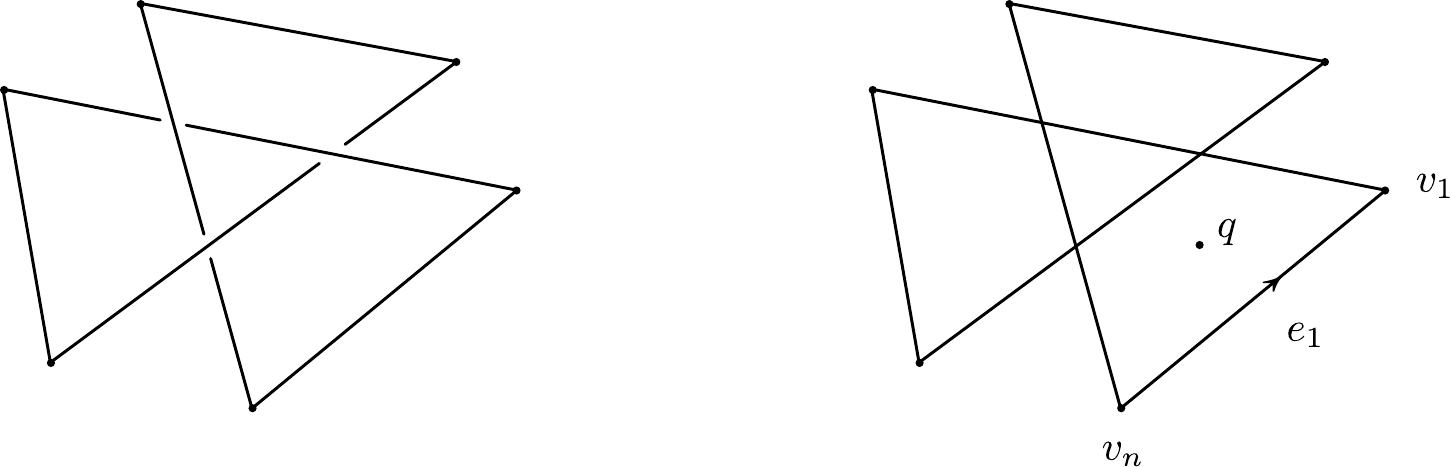}
\caption{(Left) A minimal stick diagram $s$ for the trefoil; (Right) The diagram $s$ with crossing information omitted, along with an orientation and a labeling (some labels hidden for clarity).}
\label{labeling-shadow}
\end{figure}

Now for $1\leq i\leq n$, define $p_i$ as the line segment with endpoints $v_i$ and $q$.
Let $\theta_i$ denote the angle between $p_i$ and $p_{i-1}$ (or, in the case that $i=1$, between $p_1$ and $p_n$), as displayed in Figure \ref{finding-q}, left.

We claim there exists a transformation of $q$ along the $z$-axis such that 
\begin{equation}
\sum_{i=1}^n \theta_i = 2\pi. \label{eq:angle-eq}
\end{equation}
First observe that $\sum_{i=1}^n \theta_i \geq 2\pi$.
This is because $q$ is located within in a region bounded by $s$ and the angles from edges involved in bounding this region alone must sweep out at least $2\pi$ radians. 
If $\sum_{i=1}^n \theta_i = 2\pi$ the claim is proved;
otherwise $\sum_{i=1}^n \theta_i > 2\pi$.
In this case observe that as $q$ is shifted along the $z$-axis to $z=\infty$, each $\theta_i$ becomes arbitrarily small.
This narrowing of the angles is continuous, so by the intermediate value theorem there exists some $z$ value for which $q$ satisfies \eqref{eq:angle-eq}.
We refer to the configuration of line segments with $q$ satisfying \eqref{eq:angle-eq} as the construction $Q$.

\begin{figure}
\centering
\includegraphics[width=.9\textwidth]{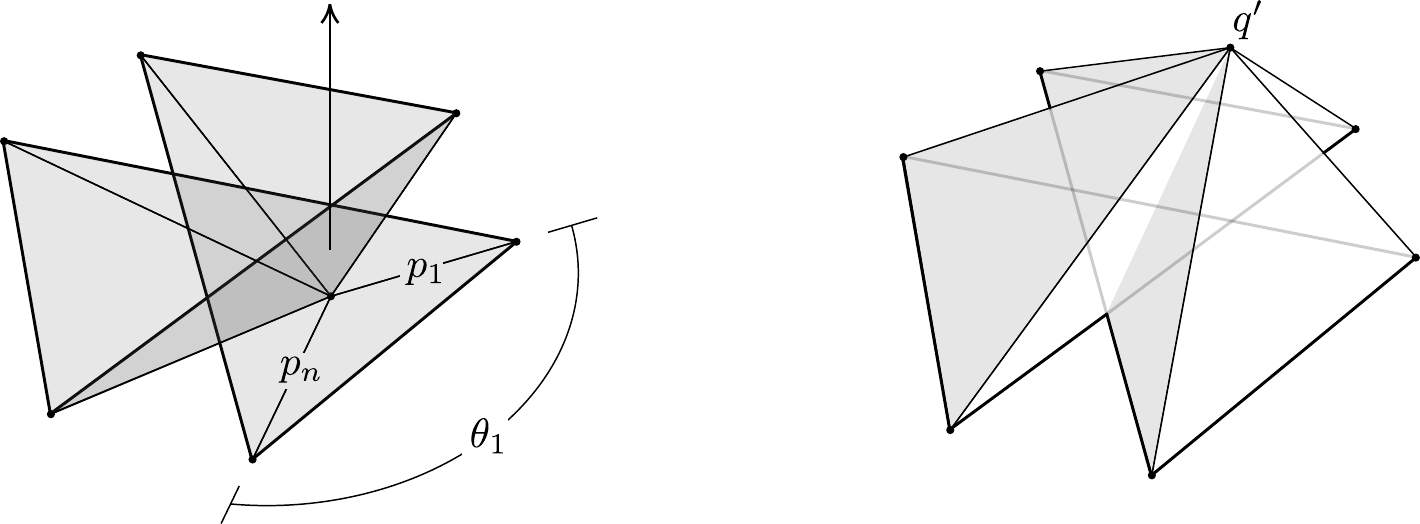}
\caption{(Left) A partial example of segments $p_i$ and the $\theta_i$ between them. Point $q$ is about to be shifted along the $z$-axis; (Right) The appropriate $z$ values for $q$ has been determined such that $\sum \theta_i = 2\pi$.}
\label{finding-q}
\end{figure}

We are now prepared to define $F$ and $\ell$, beginning with a construction of the crease pattern for $F$.
Starting with the square $[0,1]^2$, place a point $q'$ at $(\sfrac{1}{2},\sfrac{1}{2})$.
Now, moving in a circle about $q'$, place points $\{v'_i\}_{i=1}^n$ such that the triangles $q'v'_iv'_{i-1}$ are congruent to the triangles $qv_iv_{i-1}$ in $Q$.
Note that this is possible because \eqref{eq:angle-eq} is satisfied.
(If this requires any of the $v'_i$ to be outside of the paper, rescale the original knot diagram so this is no longer the case.)
Now extend the line segments $\{\overline{q'v'_i}\}_{i=1}^n$ to the edge of the paper.
Let these define the crease pattern for $F$ (see Figure \ref{final_isometry}, left , for an example).

To construct $F$ itself, begin by mapping each triangle $q'v'_iv'_{i-1}$ in the square to its counterpart $qv_iv_{i-1}$ in construction $Q$.
The mapping of each triangle is a rigid transformation in $\mathbb{R}^3$ and adjacencies between triangles are also preserved.
Hence this mapping of triangles is an arcwise isometry.
$F$ may then be linearly extended to the remaining regions of the square as depicted in Figure \ref{final_isometry}, right.

Letting $\ell$ be a loop with the image $\bigcup_{i=1}^n e_i$, we have produced an $F$ and $\ell$ such that $F(\ell) = s$.
We may then replace $\epsilon$-small regions of the $e_i$s with piecewise-linear approximations of semicircles to recover the over- and under-crossings of the original knot $K$, as demonstrated in Figure \ref{final_isometry}, right.
To complete the proof, observe that there are $n$ creases so the fold number of a knot is at most its diagram stick number.
\begin{figure}
\centering
\includegraphics[width=\textwidth]{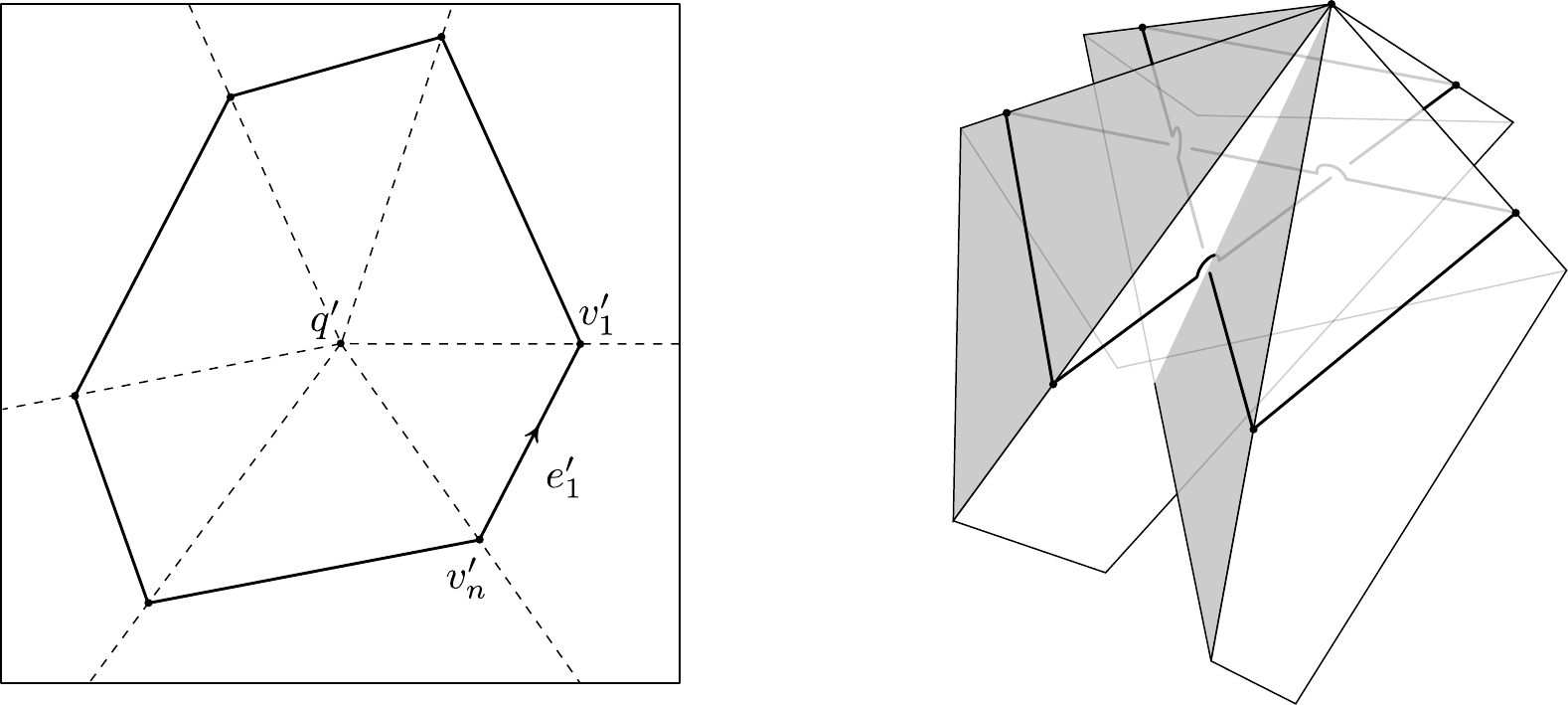}
\caption{(Left) The crease pattern for $F$ with $e_1$ for reference; (Right) The final isometry, including the small alterations to $\ell$ in order to recover the appropriate crossings.}
\label{final_isometry}
\end{figure}
\end{proof}

This upper bound is far from tight some some classes of knots, however:

\begin{theorem}
\label{thm:torus}
The fold number of a $(2,2n+3)$-torus knot ($n\geqslant0$) is two.
\end{theorem}
\begin{proof}
First we show the fold number is at most two via a construction given in Figure \ref{torus-knot}.
A loop $\ell$ is embedded on a rectangle as shown in (a).
When folded according to the crease pattern, crossings are created for each of $n+1$ pairs of ``teeth'' in the loop.
Re-arranging $F(\ell)$ shows a torus knot with $2n+3$ crossings.
The result then follows from Corollary \ref{foldnumatleasttwo}, wherein we show that the fold number for any nontrivial knot is at least two.
\begin{figure}
\centering
\includegraphics[width=.8\textwidth]{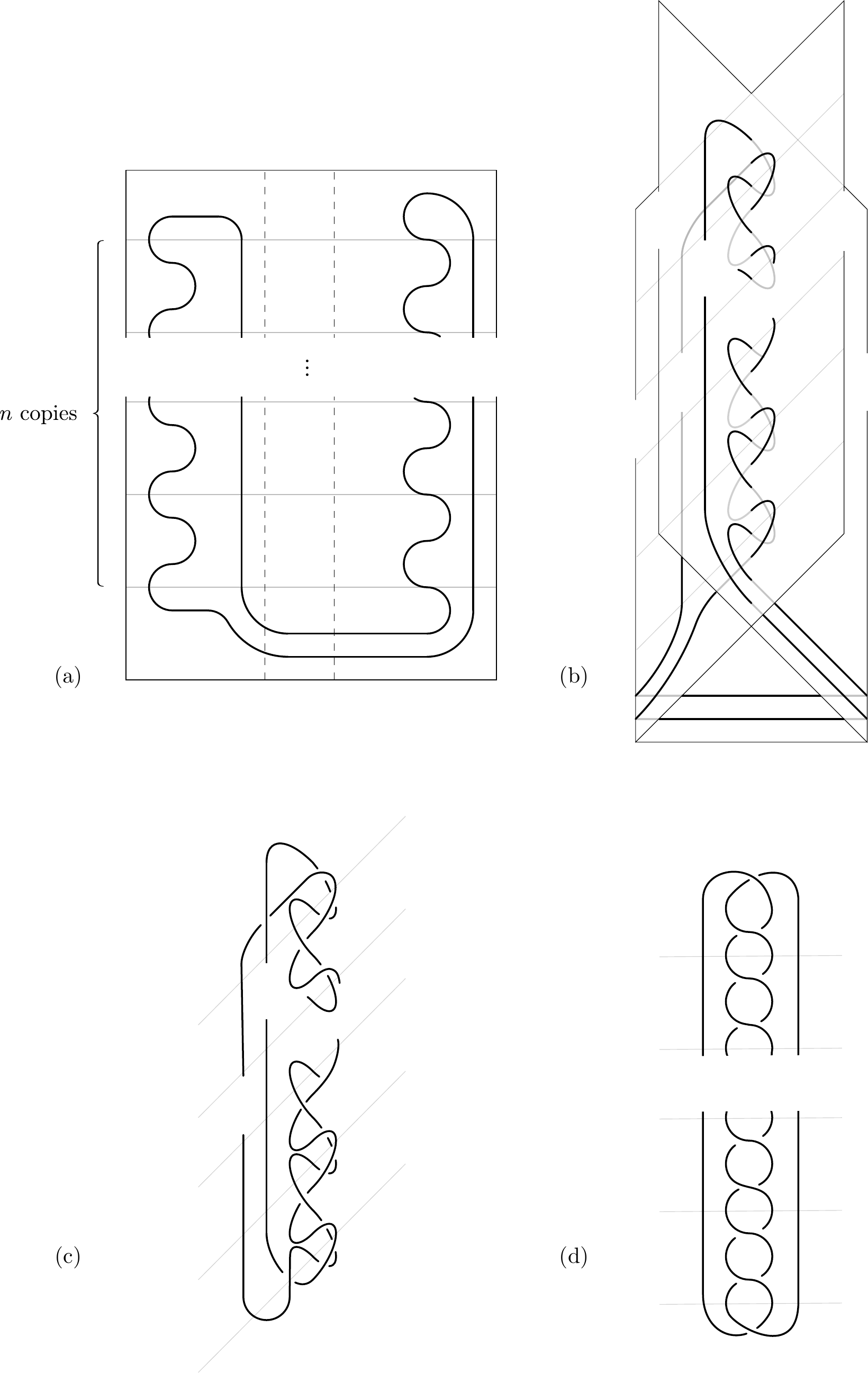}
\caption{(a) Embed a loop on a rectangle of paper. (b) Fold according to the given crease pattern. (c) Loop with paper removed. (d) Use Reidemeister moves to remove superfluous crossings.}
\label{torus-knot}
\end{figure}
\end{proof}

There is much left to be determined about the fold number.
We know it is defined for all knots and that it is upper-bounded by another well-known knot invariant, but it remains to be seen whether the fold number is a \emph{useful} invariant.
For instance, it may be that the fold number for all knots is bounded by a constant and that all knots are admitted by a construction similar to that in Figure \ref{torus-knot}.
This would suggest the complexity of knots can be completely shifted into the geometry of the loop $\ell$ itself.

\section{Proper Foldings}
\label{sec:proper-foldings}

Another project related to the production of knots by foldings is the study of the set of foldings which admit nontrivial knots in the first place.
Let $\mathcal{F}$ denote this set and $\mathcal{F}^c$ denote its complement (those foldings which only admit the unknot).

For instance, we would expect that any folding which is ``physically realizable'', or can be folded out of a real sheet of paper, is in $\mathcal{F}^c$.
We will refer to these physically-realizable maps as \emph{proper} and in this section confirm the above intuition.
(Note that this fact is not immediate because certain foldings we would like to designate as physically-realizable---such as flat foldings---are not injective.)

On the other hand, we will show that there exist non-physically-realizable, or \emph{improper}, foldings which also do not admit nontrivial knots.
Therefore (im)properness does not completely characterize $\mathcal{F}$.

We begin by formalizing the intuition of physical-realizability.
We would like our definition to describe the possible configurations of an idealized zero-thickness sheet of paper and as a result we require a ``non-crossing'' condition similar to that described in \cite{Demaine2007GeometricPolyhedra}.
In particular, we wish to disallow material crossings of the paper while still allowing for geometric contact of layers of the paper, as exemplified in Figure \ref{improper-foldings}.

\begin{definition}
An origami folding $F$ is \emph{proper} if for all $\epsilon>0, $ there exists an injective map $F'$ such that $|F-F'|<\epsilon$ according to the supremum norm.
Otherwise, we say $F$ is \emph{improper}.
\end{definition}

\begin{figure}
\centering
\includegraphics[width=\textwidth]{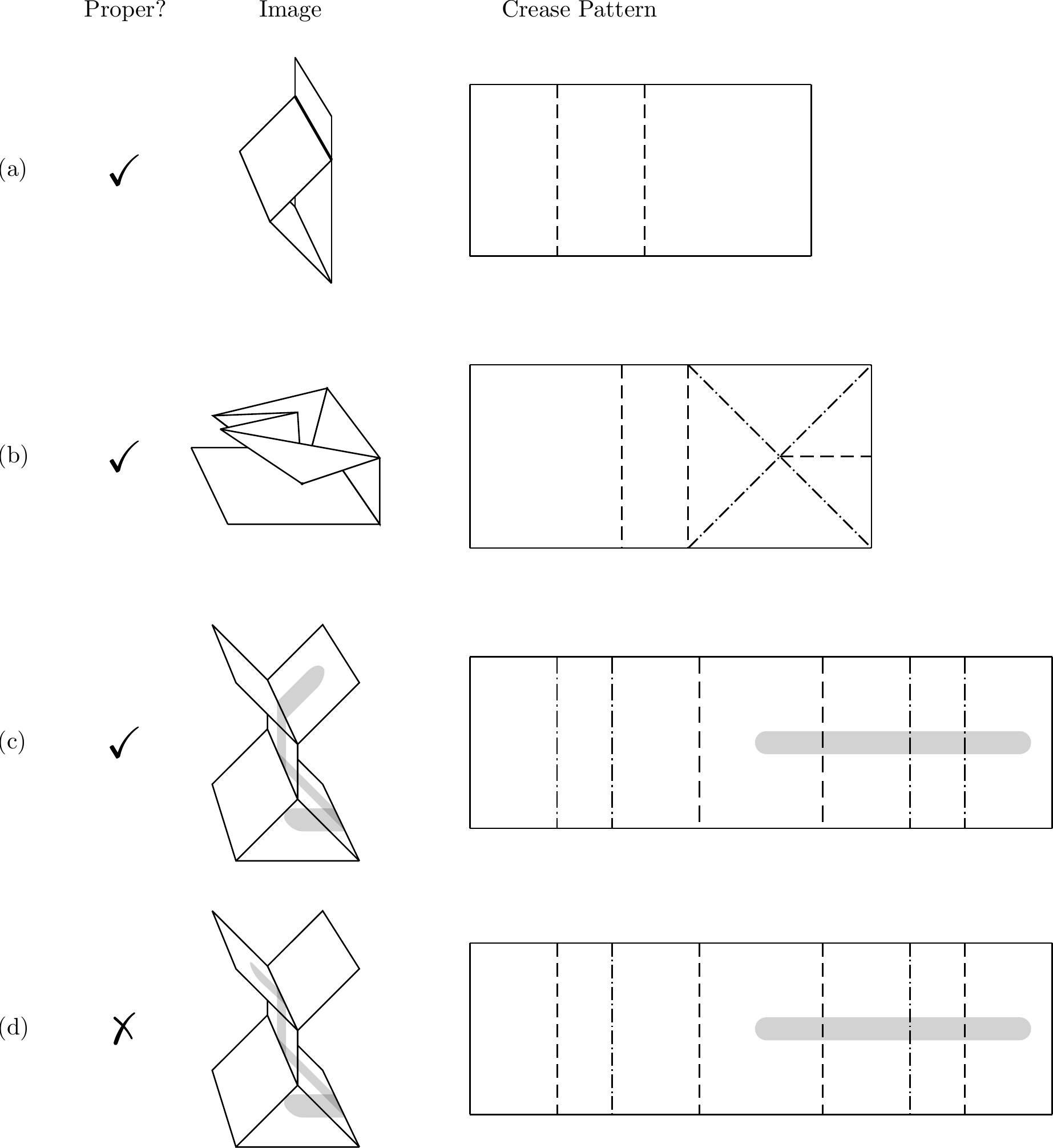}
\caption{(a) A proper folding that self-intersects along a line; (b) A proper folding that self-intersects at a point; (c) A proper folding that self-intersects on a plane; (d) An improper folding that self intersects on a plane. (The gray strips are included to differentiate between (c) and (d).)}
\label{improper-foldings}
\end{figure}

In other words, we define physically-realizable foldings as those which can be well-approximated by injective maps.
Note that these approximating maps do not bear the same isometry condition as the folding itself.
We claim this captures the spirit of origami as practiced in the real world: often folding instructions and crease patterns suggest the paper ought to be mapped to itself (consider, for instance, any flat-folded model), yet when working with real paper the map must be shifted slightly to produce a folded model.
This shifting is captured here by use of arbitrarily-close injective maps, which are allowed to stretch, bend, and contort within the limits set by the supremum norm.

We now prove that proper maps do not admit nontrivial knots in two stages.
First we show that injective piecewise-linear maps never admit nontrivial knots because any loop in $[0,1]^2$ can be continuously slid down into a single planar face of the image of the paper (this sliding can then be extended to the requisite isotopy of $\mathbb{R}^3$).
Second we show $\mathcal{F}^c$ is open and apply the definition of proper.

\begin{theorem}
Proper foldings do not admit nontrivial knots. (However, the converse is false.)
\label{thm:propfoldnoknots}
\end{theorem}

\begin{lemma}
Injective piecewise-linear maps do not admit nontrivial knots.
\label{lem:injective-no-knots}
\end{lemma}
\begin{proof}
Let $F$ be an injective piecewise-linear map and $\ell$ be a piecewise-linear simple curve in $[0,1]^2$.
Let $I:[0,1]^2\to \mathbb{R}^3$ be defined by $(x,y)\mapsto(x,y,0)$.
We will prove the result by exhibiting an ambient isotopy of $\mathbb{R}^3$ which maps $F(\ell)$ to a linear transformation of $I(\ell)$.

\begin{figure}
\centering
\includegraphics[width=.9\textwidth]{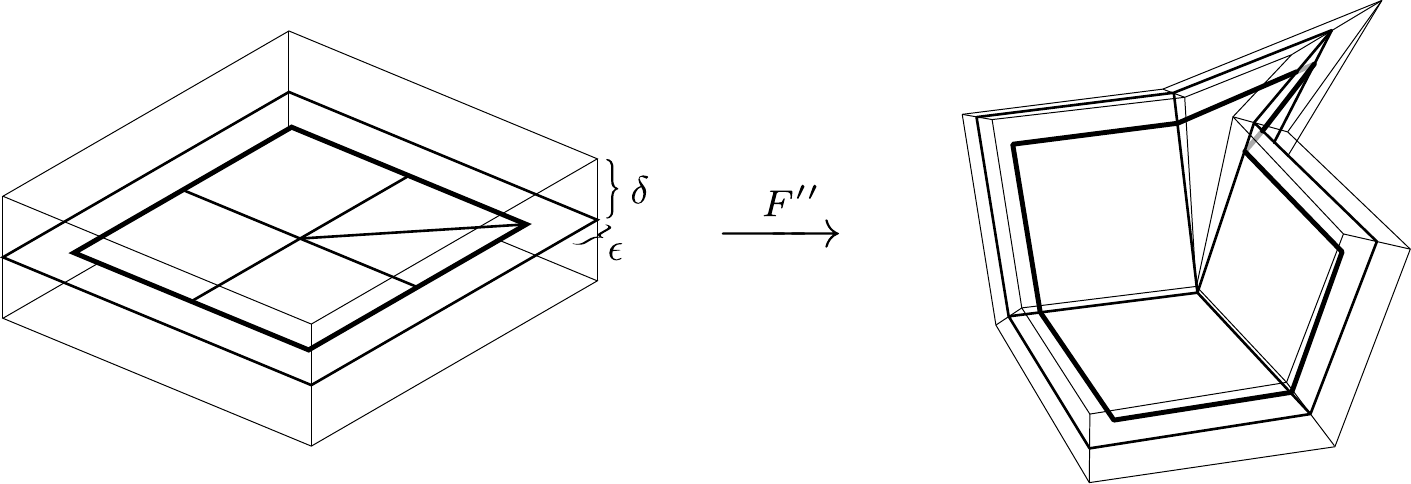}
\caption{The extension of the map $F$ to a map of the ball $F''$.}
\label{folding-to-box-map}
\end{figure}

Choose an $\epsilon$ such that the linear extension $F':[-\epsilon, 1+\epsilon]^2\to\mathbb R^3$ of $F$ remains injective (this is always possible because the paper includes its boundary).
Now define the extension $F''_{\delta}:[-\epsilon,1+\epsilon]^2\times[-\delta,\delta]\to\mathbb R^3$ as pictured in Figure \ref{folding-to-box-map}.
In particular, for each enclosed polygonal face $A$ in the crease pattern, linearly map the $\delta$-tall prisms above and below the paper to truncated prisms in the image.
These truncated prisms are defined as follows:
the prism is $\delta$-tall, and for each edge $e$ bordering $A$, the angle formed between the side of the prism attached to $e$ and $A$ is half that of the angle formed between $A$ and the other polygonal face with which it shares $e$.
It is always possible to choose a $\delta$ such that $F''_\delta$ is injective.
Fix such a $\delta$, let $B = [-\epsilon, 1+\epsilon]^2\times[-\delta,\delta]$, and define the resulting map $B\to \mathbb R^3$ as $F''$.

We now describe an isotopy $G$ on $B$ which we will lift to $\mathbb{R}^3$ via $F''$.
Begin by selecting a point $p$ on the interior of a face $C$ of the crease pattern of $F$.
Then choose an open regular neighborhood $U\subset [0,1]^2\times\{0\}$ of $p$ which is a subset of the interior of $C$.
Let $B' = [0, 1]^2\times[-\delta/2,\delta/2]$ and define $G'_t:B'\times [0,1]\to B'$ to be an isotropic (uniform) scaling of $B'$ about $p$, parametrized by $t$ in such a way that at $G'_0$ is the identity and $G'_1([0,1]^2\times \{0\})\subset U$.
We may now extend $G'_t$ to an isometry $G_t$ of all of $B$ by linearly mapping each prism in $B\backslash B'$ as depicted in Figure \ref{initial-homotopy}.

\begin{figure}
\centering
\includegraphics[width=.9\textwidth]{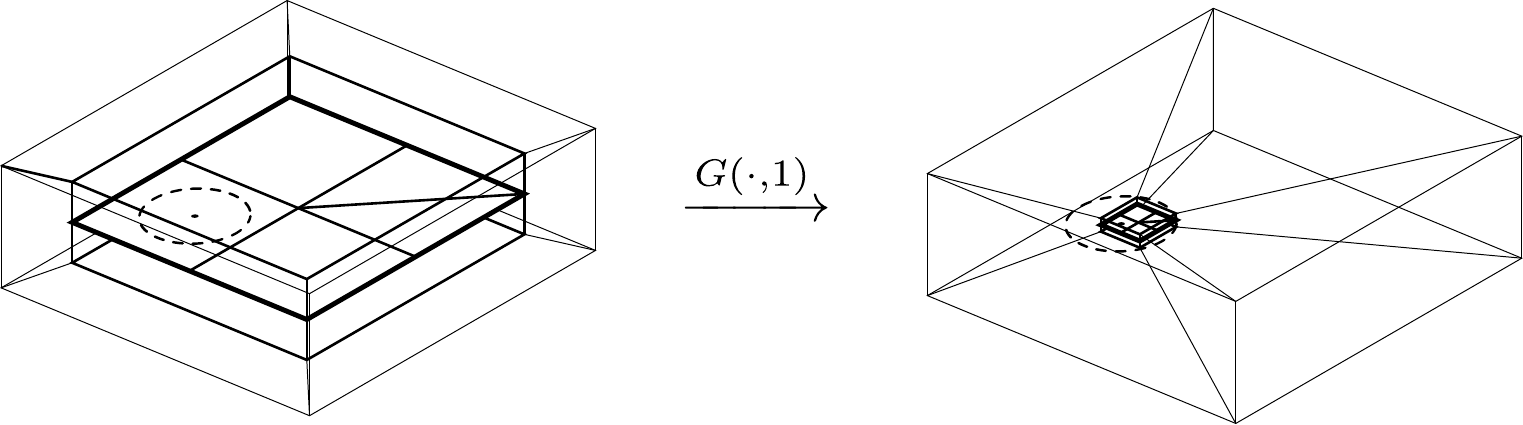}
\caption{The isotopy of $B$.}
\label{initial-homotopy}
\end{figure}

Therefore $H_t(\mathbf{x}):= F''(G_t(F''^{-1}(\mathbf{x})))$ is an isotopy of $\im(F'')$.
Note that $H_t$ restricted to the boundary of $\im(F'')$ is the identity, so we may trivially extend $H_t$ to an isotopy $H'_t$ of $\mathbb R^3$.
Observe that $H'_1(F(\ell))=F(G_1(\ell))$, \emph{i.e.,} is a composition of linear transformations of a simple closed curve, so $H'_1(F(\ell))$ is the unknot.
Therefore $H'_0(F(\ell))=F(\ell)$ is equivalent to the unknot.
\end{proof}

\begin{proof}[Proof of Theorem \ref{thm:propfoldnoknots}]
Let $\mathcal{G}$ be the set of piecewise-linear maps $[0,1]^2\to \mathbb R^3$ which admit nontrivial knots in the origami knot sense; that is, letting $G\in \mathcal{G}$, $G(\ell)$ is a nontrivial knot for some $\ell$).
We claim $\mathcal{G}$ is open in the space of piecewise-linear maps subject to the supremum norm.
Suppose again $G\in \mathcal{G}$.
Then there is a piecewise-linear loop $\ell$ such that $G(\ell)$ is a nontrivial knot.
Consider the tubular region $T$ of $\ell$ of radius $\epsilon$.
Let $G':[0,1]^2\to\mathbb R^3$ be a map such that $|G-G'|< \epsilon$.
If $G'(\ell)$ is injective set $\ell' = \ell$.
Otherwise, define $\ell'$ by perturbing $\ell$ such that $G(\ell')$ is injective and $\im(G(\ell')$) remains in $T$.
Then $G'(\ell')$ is satellite knot with a nontrivial companion knot $G(\ell')$ and is thus nontrivial.

Now suppose $P$ is a proper folding.
By definition there exists a sequence of injective maps $G_1, G_2, \ldots$ which converge to $P$.
But piecewise-linear maps are dense in the space of continuous maps $[0,1]^2\to \mathbb{R}^3$ subject to the supremum norm, so there exists a sequence of piecewise-linear injective maps $G_1', G_2', \ldots$ which also converge to $P$.
Each $G_i$ is in $\mathcal{G}^c$ by Lemma \ref{lem:injective-no-knots} and we have shown above that $\mathcal{F}^c$ is closed.
Therefore $P$ is also in $\mathcal{G}^c$ and the proof of the first statement is complete.

We now construct an improper folding which does not admit a nontrivial knot.
\begin{figure}
\centering
\includegraphics[width=\textwidth]{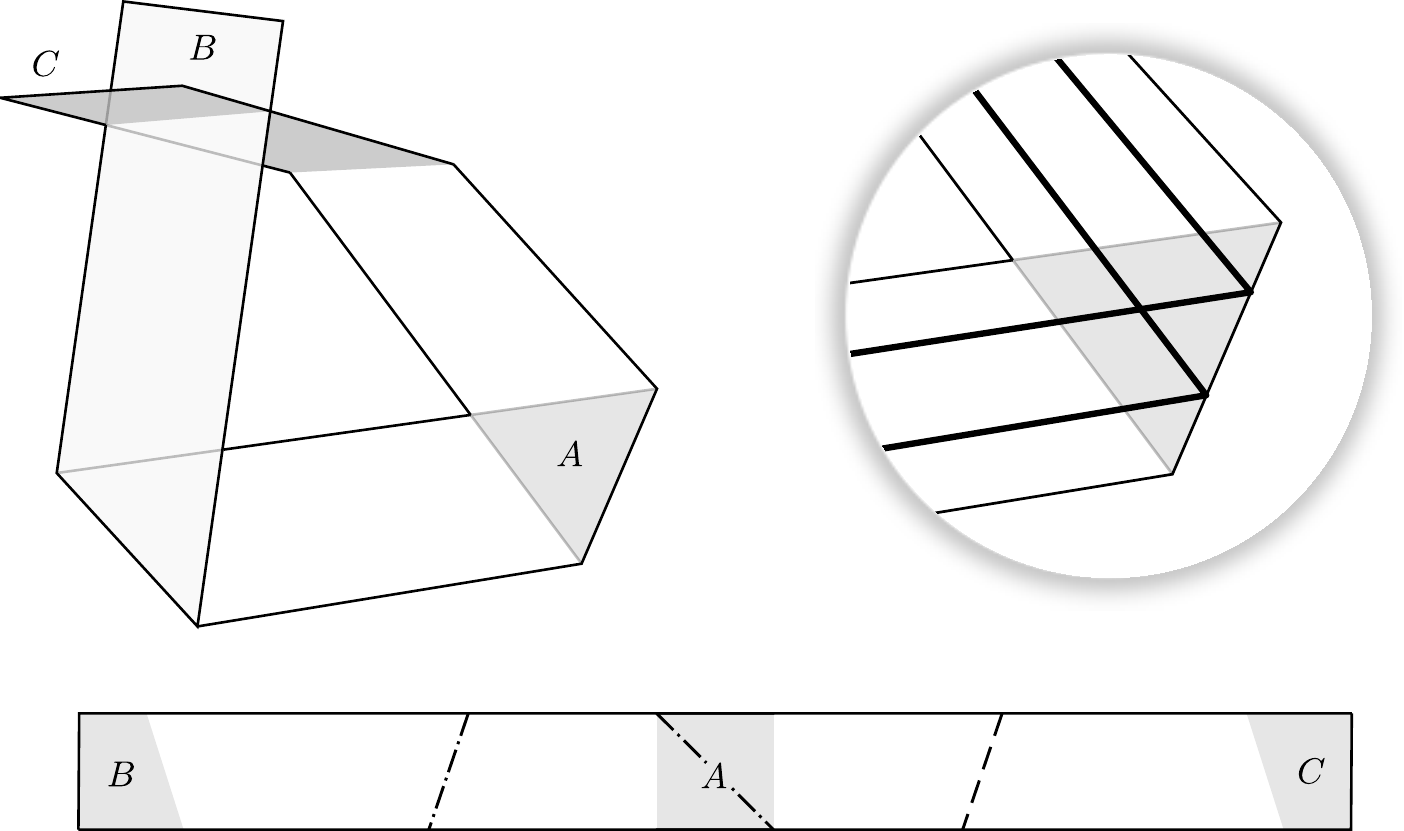}
\caption{An improper folding (left) which admits no nontrivial knots, along with a detail (right) and its crease pattern (bottom). The square region $A$ is mapped to a triangle in the image and the piece of paper intersects transversally at the boundary of regions $B$ and $C$.}
\label{improper-no-knot}
\end{figure}
Consider the construction in Figure \ref{improper-no-knot}.
Here the triangular halves of region $A$ are identified to each other and the ends of the paper self-intersect transversally.
Note that the folding is improper but that, if restricted to the complement of $B$ or the complement of $C$, the folding becomes proper.
Thus, were there to exist a loop whose image is a nontrivial knot, it would have to intersect both region $B$ and region $C$.
However, any loop $\ell$ which intersects both of these regions must also pass through region $A$ in both directions.
Because one half of region $A$ is identified to the other along a diagonal, there is no such loop $\ell$ whose image does not have a self-intersection analogous to that pictured in Figure \ref{improper-no-knot}, right.
Therefore all loops which intersection $B$ and $C$ also have a self-intersection in the region $A$ and therefore do not form knots as their image.
\end{proof}

\begin{corollary}
The fold number of a knot $K$ is zero if $K$ is the unknot; otherwise the fold number of $K$ is at least two.
\label{foldnumatleasttwo}
\end{corollary}
\begin{proof}
The first statement is trivial.
To address the latter, observe that the crease pattern of a folding $F$ with one edge defines two regions of the paper separated by a crease.
Therefore $F$ is, without loss of generality, the rotation by an angle $\theta < 2\pi$ of one component of the paper about the axis defined by the crease.
If $\theta \neq \pi$, $F$ is injective and therefore proper.
If $\theta = \pi$ there is a sequence of injective maps (\emph{e.g.,} with $\theta_n = \pi-1/n$; $n\to \infty$) converging to $F$ and thus $F$ is proper.
By Theorem \ref{thm:propfoldnoknots} $F$ therefore does not admit nontrivial knots.
\end{proof}

Because the converse is false, properness does not completely characterize foldings which do not admit nontrivial knots.
This is appears to be due to the fact that properness is determined by a local condition (a small portion of a folding can render a map improper), while the embedability of a nontrivial knot is, as we see in the counterexample, a global property.

\section{Further Research}
\label{sec:conclusions}
We conclude by summarizing the questions raised in the body of the paper, as well as proposing a few more.

\subsection{Further questions about the fold number}
\begin{openq}
Is there a ``nice' property of foldings which completely characterizes which foldings admit knots?
\end{openq}

\begin{openq}
What can be said about origami links? (That is, begin with multiple non-intersecting simple loops in the square and examine foldings which intertwine them.)
\end{openq}

\begin{openq}
Bounds: is the fold number bounded below by other knot invariants?
Is the fold number bounded above by a constant?
What fold numbers can be determined for other knot classes?
\end{openq}

\begin{openq}
How does the fold number factor over the knot sum?
\end{openq}

\begin{openq}
How is the fold number invariant affected by a change in the topology of the original origami paper? (\emph{E.g.,} foldings from a M{\"o}bius strip, cylinder, torus, etc.)
\end{openq}

\subsection{Alterations to the formalism}

The formalization of folding and properness used above represents an attempt to balance topological brevity with faithfulness to origami.
However, there are other interesting definitions and corresponding invariants nearby---both those which prefer topological simplicity and those which more-closely match definitions of origami folding extant in the literature.

In particular, moving towards pure topology, questions remain if we remove the isometry condition of foldings:
\begin{openq}
Call a map $f:[0,1]^2\to \mathbb R^3$ a \emph{weak folding} if it is piecewise linear and define \emph{weak origami knot} and \emph{weak fold number} analogously.
Certainly the weak fold number is defined for all knots and is bounded above by the (strong) fold number.
But what about other bounds?
Are the weak and strong fold numbers linearly (or polynomially) related, or is there a larger gap?
\end{openq}

In the mathematics of origami direction, one component of origami folding which is absent from our study thus far, but which is present in the formalism initiated in \cite{Justin1997TowardOrigami} and refined in Chapter 11 of \cite{Demaine2007GeometricPolyhedra}, is a notion of layer ordering.
That is, foldings which map faces of a crease pattern to other faces ought to specify exactly how these faces are to stack together, were an origamist to go and fold it.

The existance of physically-realizable layer orderings for a folding is already captured by the notion of properness presented above, and indeed all such orderings are the same in the eyes of knot theory (they admit only the unknot, by Theorem \ref{thm:propfoldnoknots})
However, our formalism does not capture layer orderings which \emph{do not} satisfy the non-crossing condition.
With enough care knot embeddings may be defined for these self-crossing layer orderings, representing another interesting environment for the development of knot invariants.

We will be brief and informal here, leaving a fully rigorous exposition for later work.

\begin{definition}
Suppose there exists a folding and a layer-ordering pair $(F,\lambda)$ (which we will call a \emph{layered folding}), as well as a loop $\ell$ on the paper such that $F(\im{\ell})$ is non-injective only on isolated points in the interior of crease pattern faces which are mapped into other faces. (That is, nonzero-length segments of $S^1$ may not be mapped to other nonzero-length segments.)
Then define the map $L$ which is equal to $F(\ell)$ everywhere except for epsilon-small balls about these points of non-injection.
In these balls, adjust the image of $S^1$ according to $\lambda$, so that the curve pieces are separated as prescribed by the layer ordering.
Then $L$ is injective.
We say $K$ is a \emph{layered origami knot} if there exists such a $(F,\lambda)$ and $L$ such that $L=K$.
\end{definition}

This definition gives rise to two more directions for further study: general two-dimensional layered origami folding and flat-folded origami.

\begin{definition}
The \emph{layered fold number} of a knot $K$ is the minimum number of creases required by a layered folding which admits a knot equivalent to $K$
\end{definition}

Many of the above results carry through directly to this setting:
the folding constructed in Theorem \ref{thm:allknots} can be extended trivially to a layered folding, and the torus knots of Theorem \ref{thm:torus} may still be constructed in the same way.
Thus the layered fold number is bounded above by the fold number.
However, the construction disproving the converse to Theorem \ref{thm:propfoldnoknots} is no longer valid. In fact, we conjecture there is no counterexample:

\begin{conjecture}
A layered origami folding admits a nontrivial layered origami knot if and only if it is improper.
\end{conjecture}

\noindent This would settle a portion of the open question in \cite{Justin1997TowardOrigami} concerning the unknottedness of the boundary of a proper folding.

Finally, we may define another invariant by restricting to flat foldings:
\begin{definition}
The \emph{flat-folding number} of a knot $K$ is the minimum number of creases required by a layered flat folding (mapping into $\mathbb{R}^2$) which admits a knot equivalent to $K$.
\end{definition}

Note that all knots may be flat folded by beginning with the construction in Theorem \ref{thm:allknots} and adding reverse-folds between the existing creases as necessary to flatten the self-intersecting cone.
Therefore many of the questions posed above have analogous questions in this largely combinatorial setting.
Beyond questions posed above, it may also be interesting to explore a connection between Reidemeister moves and a flat folding sequence.

\section*{Acknowledgements}
We thank Carleton College for its financial and academic support of an early version of this project.
We are indebted to Deanna Haunsperger and Helen Wong for their supervision and guidance, as well as to Thomas Hull and Derek Sorenson for helpful discussions.

\bibliographystyle{osmebibstyle}
\bibliography{osmerefs}

\begin{thebibliography}{99}
\newcommand{\enquote}[1]{``#1''}
\newcommand{\doi}[1]{doi:#1}

\bibitem[Alperin and Lang~06]{Alperin2006One-Axioms}
Roger~C. Alperin and Robert~J. Lang.
\newblock \enquote{{One-, two, and multi-fold origami axioms}.}
\newblock In \emph{Proceedings of 4th International Conference on Origami,
  Science, Mathematics, and Education}. Pasadena, CA: 4OSME, 2006.

\bibitem[Burago et~al.~01]{Burago2001AGeometry}
Dmitri Burago, Yuri Burago, and Sergei Ivanov.
\newblock \enquote{{A Course in Metric Geometry}.}
\newblock \emph{Graduate Studies in Mathematics}.
\newblock \doi{10.1090/gsm/033}.

\bibitem[Demaine and O'Rourke~07]{Demaine2007GeometricPolyhedra}
Erik~D. Demaine and Joseph O'Rourke.
\newblock \emph{{Geometric Folding Algorithms: Linkages, Origami, Polyhedra}}.
\newblock Cambridge University Press, 2007.

\bibitem[Justin~91]{Justin1991ResolutionGeometriques}
Jacques Justin.
\newblock \enquote{{Resolution par le pliage de l'equation du troisieme degre
  et applications geometriques}.}
\newblock In \emph{Proceedings of the First International Meeting of Origami
  Science and Technology}, edited by {H. Huzita}, pp.~251--261. Padova, Italy:
  Dipartimento di Fisica dell'Universit{\`{a}} di Padova, 1991.

\bibitem[Justin~97]{Justin1997TowardOrigami}
Jacques Justin.
\newblock \enquote{{Toward a mathematical theory of origami}.}
\newblock In \emph{Origami Science and Art: Proceedings of the Second
  International Meeting of Origami Science and Scientific Origami}, edited by
  K.~Muria, pp.~15--29. Otsu: Seian University of Art and Design, 1997.

\end{thebibliography}
{\parindent0pt
\small
    \vskip 1em \hrule \vskip 1em 

Joseph Slote\\[-.8em]

\parshape 1 1em \dimexpr\linewidth-1em\relax
California Institute of Technology, 1200 E. California Blvd., Pasadena, CA 91125,
\email{jslote@caltech.edu}\\[-.7em]

\parshape 1 1em \dimexpr\linewidth-1em\relax
{\footnotesize\textit{At time of original publication:}}

\parshape 1 1em \dimexpr\linewidth-1em\relax
Carleton College, 1 North College Street,
Northfield, MN 55057\\[1em]

Thomas Bertschinger\\[-.8em]

\parshape 1 1em \dimexpr\linewidth-1em\relax
Los Alamos National Laboratory, Los Alamos, NM 87545,
\email{tahbertschinger@gmail.com}\\[-.7em]

\parshape 1 1em \dimexpr\linewidth-1em\relax
{\footnotesize\textit{At time of original publication:}}

\parshape 1 1em \dimexpr\linewidth-1em\relax
Carleton College, 1 North College Street,
Northfield, MN 55057

}



\end{document}